\documentclass[11pt, twoside]{article}
\usepackage{graphicx}
\usepackage{amssymb}
\usepackage{color}
\usepackage{fancyhdr}

\newtheorem{thm}{Theorem}[section]
\newtheorem{corollary}[thm]{Corollary}
\newtheorem{prop}[thm]{Proposition}
\newtheorem{defn}[thm]{Definition}

\newcounter{step}
\newcounter{claim}
\newcounter{subclaim}

\newenvironment{proof}{\textbf{Proof}\ }{\hfill $\square$}

\newenvironment{claim}{\vspace{0.75\baselineskip}
\textbf{Claim\ }}{\vspace{0.75\baselineskip}}

\newcommand{\oshs}{one-sided Heegaard splitting}
\newcommand{\oss}{one-sided splitting}
\newcommand{\tshs}{two-sided Heegaard splitting}
\newcommand{\tss}{two-sided splitting}
\newcommand{\gi}{geometrically incompressible}
\newcommand{\gc}{geometrically compressible}

\def\co{\colon\thinspace}

\begin{document}

\thispagestyle{empty}

\scalebox{2}{}

\begin{center}
\scalebox{1.5}{\textbf{Incompressible one-sided surfaces}}\\
\vspace{2mm}

\scalebox{1.5}{\textbf{in filled link spaces}}
\vspace{\baselineskip}

\textsc{Loretta Bartolini}
\vspace{\baselineskip}

\parbox{.8\textwidth}
{
\textbf{Abstract}
\vspace{.3\baselineskip}

\footnotesize{When a Dehn filled link manifold contains a \gi{} one-sided surface, it is shown there is a unique boundary incompressible position that the surface can take in the link space. The proof uses a version of the sweep-out technique from \tshs{} theory. When applied to \oshs{s}, this result can be used to complete the classification of \oss{s} of $(2p, q)$ fillings of Figure 8 knot space: determining that fillings with $|{2p \over q}| <3$ have two non-isotopic \gi{} \oss{} surfaces.}
}
\end{center}

\section{Introduction}

The study of orientable incompressible surfaces in knot complements has been an active area since Lickorish established that any closed, orientable $3$-manifold can be obtained by Dehn filling such a space~\cite{lickorish}. By understanding incompressible surfaces in knot or link spaces, which is a more tractable problem than deciding whether a closed manifold is itself Haken, one can hope to gain an informed perspective on whether incompressible surfaces persist in the filled manifold. However, the process of generalising is heavily limited by the unpredictable behaviour of orientable surfaces during Dehn filling. Both incompressible surfaces and \tshs{} surfaces become difficult to control when solid filling tori are introduced; properties of incompressibility and strong irreducibility respectively are no longer assured.

Given one-sided surfaces have a natural connection to the $\mathbb{Z}_2$-homology of a manifold, it is intuitive that such surfaces be more predictable under Dehn filling. Homological constraints effectively `pin' the surface in the link space and the solid filling tori, allowing only a limited range of movement in between. Formalising this notion using the sweep-out techniques of Rubinstein and Scharlemann~\cite{rubin_scharl}, it can be shown that any closed, \gi{}, one-sided surface in a filled link space cannot restrict to multiple non-isotopic boundary incompressible surfaces when confined to the link manifold.

This result is of particular utility when considering classes of manifolds obtained from fillings of a particular space, for which the \gi{}, boundary incompressible surfaces are known. For example, such surfaces in Figure 8 knot space are well-known through the work of Thurston~\cite{thurston}, and Dehn fillings of this space have been of significant and lasting interest, given all but the simplest handful of fillings are hyperbolic. As such, this class provides an immediate source of non-trivial examples.

The classification of \gi{} one-sided surfaces in `most' $(2p, q)$-fillings of Figure 8 knot space has been established in~\cite{nonhaken}, however this is limited to when $|{2p \over q}| >3$, hence a single minimal genus surface is present. The main result herein serves to distinguish between multiple minimal genus surfaces, thus extending the classification to the whole class.

Given the direct links between \gi{} one-sided surfaces and \oshs{s}, \oss{s} are thus fully classified for the entire class of hyperbolic $3$-manifolds obtained from even Dehn fillings of Figure 8 knot space. This is of particular note given no classification exists for \tss{s} of any closed hyperbolic $3$-manifold. Furthermore, this demonstrates a degree of control over \oss{} surfaces during Dehn filling that is not currently known to hold for \tss{s}.

\section{Preliminaries}
Throughout, let $M$ be a compact, orientable $3$-manifold and consider all manifolds and maps as PL.

\begin{defn}
A link space or manifold is a compact $3$-manifold $M$, with non-empty boundary consisting of a collection of tori. Call $M$ a knot space if its boundary is a single torus.
\end{defn}

Note that this is a generalisation of a knot or link exterior in $S^3$.

\begin{defn}
A one-sided surface $K \subset M$ is \gi{} if any simple, closed, non-contractible loop on $K$ does not bound an embedded disc in $M \setminus K$. Call $K$ \gc{} if it is not \gi{}.
\end{defn}

\begin{defn}
A one-sided surface $K$ in a bounded $3$-manifold $M$ is boundary compressible if there exists an embedded bigon $B \subset M$ with $\partial B = \alpha \cup \beta$, $B^{\circ} \cap K = \emptyset$, where $\alpha \subset K, \beta \subset \partial M$ and $\alpha$ is not homotopic into $\partial M$. Call $K$ boundary incompressible in $M$ if no such bigon exists.  
\end{defn}

Note that since $K$ is one-sided, if $K \cap \partial M$ is essential on $K$, a boundary compressing bigon can correspond to the boundary compression of a M\"obius band. In this case, $\beta$ has ends on locally opposite sides of $\partial K$ and the move is referred to as a {\it M\"obius band compression}. If $\beta$ has ends on locally the same side of $\partial K$, this corresponds to a boundary compression in the usual sense of orientable surfaces and is referred to here as an {\it orientable boundary compression}. Note that if $K$ is \gi{}, a M\"obius band compression is the only form a boundary compression can take, as an orientable move would correspond to a geometric compression of the surface.

\begin{prop}[\cite{nonhaken}]\label{prop:boundary_incomp}
Any \gc{} one-sided surface with non-empty boundary in a link manifold, which has no essential annuli, can be arranged into a set of M\"obius bands in a neighbourhood of the boundary, and a boundary incompressible component in the complement, which is unique up to isotopy.
\end{prop}

This is an important result, establishing a canonical form for \gc{} one-sided surfaces in link manifolds.

\section{Incompressible one-sided surfaces in filled link\\ spaces}

\begin{thm}\label{thm:unique_position}
A \gi{} one-sided surface embedded in an even Dehn filling of a link manifold, where the surface is dual to the cores of the solid filling tori and the link space has no essential annuli, has a unique boundary incompressible position in the link space.
\end{thm}

\begin{proof}
Take an irreducible Dehn filling $M$ of the link manifold $M_0$ that has no essential annuli, with $\{M_{T_k}\} = \overline{M \setminus M_0}$ the collection of solid filling tori. Let $K, L \subset M$ be embedded, \gi{} one-sided surfaces of the same genus.

\begin{claim}
If $K$ and $L$ boundary compress across $\{T_k\}$ to different boundary incompressible surfaces in $M_0$, then they are non-isotopic in $M$.
\end{claim}

In the filled manifold, boundary compressions across the boundaries of the solid filling tori can take two forms: non-orientable and orientable. The non-orientable compressions correspond to pushing M\"obius bands back and forth across a torus; a process that can continue indefinitely. Here, surfaces are to be considered as boundary incompressible with respect to the link space, therefore such compressions are completed maximally out of $M_0$. Orientable boundary compressions correspond to a loop of intersection that bounds a trivial disc on the surface, since $K, L$ are \gi{}; this disc may be on either side of a torus $T_k$. Therefore, orientable compressions occur in both directions across the tori, however, the process terminates after a finite number of steps, after which all trivial loops of intersection are removed and the surface is `orientable boundary incompressible' from either side. Note that orientable compressions only affect non-orientable compressions trivially.

In the link space $M_0$, boundary compress $K, L$ across the boundary of the solid filling tori to get boundary incompressible surfaces $K_0, L_0$, respectively, that are non-isotopic in $M_0$.

Suppose, by way of contradiction, that $K$ and $L$ are isotopic in the filled manifold $M$. Then there exists an isotopy $\varphi \co M \times [0, 1] \rightarrow M$ such that $\varphi(K \times \{0\}) = K$ and $\varphi(K \times \{1\}) = L$. Parametrise $\varphi$ by $t \in [0, 1]$, such that $\varphi_{t}(K) = \varphi(K \times \{t\})$.

Putting the $\varphi_t$ in general position relative to each component $T_k = \partial M_{T_k} \subset \partial M_0$ ensures the inverse images $\varphi_t^{-1}(T_k)$ are smooth surfaces in the domain $K \times [0,1]$. Moreover, for finitely many values $t_i$ of $t$, the inverse images $\varphi_{t_i}^{-1}(T_k)$ are tangent to $K \times \{t_i\}$. These tangencies can be either maxima or minima or of saddle type. Call the points $t_i \in (0, 1)$, and the surfaces corresponding to them, {\it critical}.

Subdivide $[0, 1]$ into intervals bounded by critical points: $0 < t_1 < \ldots < t_n < 1$. Extending this subdivision, $K \times I$ is decomposed into regions, bounded by critical surfaces, on the interior of which $\varphi_t(K)$ and $\partial M_0$ are transverse and intersect in a collection of simple closed curves.

Characterise the curves of intersection between each $T_k$ and the image of $K$ by whether they bound discs on the boundary of the solid filling tori. By the geometric incompressibility of $K$, curves that are essential on the image of $K$ do not bound discs in $M$, hence are essential on both $T_k$ and $M_{T_k}$.

The existence of intersection curves that are essential on each $T_k$ is determined by homology. The core $\alpha_k$ of a solid filling torus $M_{T_k}$ represents a non-trivial class in $H_1(M; \mathbb{Z})$, therefore $H_1(M;  \mathbb{Z}_2)$, as the Dehn filling is even. Since $K$ is \gi{}, it represents a non-trivial $\mathbb{Z}_2$-homology class. As the first and second $\mathbb{Z}_2$-homology classes are dual by Poincar\'e duality, the respective representatives of $\alpha_k, K$ have intersection number one mod $2$. Hence, $K$ intersects every $T_k$ in an odd number of homologically non-trivial curves. All such curves are essential on $T_k$, yet may be trivial on $K$.

\subsection{Labelling regions}

As $K$ sweeps across $K \times I$ under the isotopy $\varphi_t$, the curves of intersection with each torus $T_k$ change as the surface passes across it. Label regions $K \times (t_i, t_{i+1})$ as $\mathcal{K}$ (or $\mathcal{L}$) if the corresponding image of $K$ boundary compresses to $K_0$ (or $L_0$) in $M_0$.

Note that no region carries both $\mathcal{K}$ and $\mathcal{L}$ labels, as a bounded, \gi{}, one-sided surface in a link manifold with no essential annuli boundary compresses to a unique boundary incompressible surface, by Proposition~\ref{prop:boundary_incomp}.

Since $M$ contains no essential annuli, the behaviour of $K$ at any one boundary torus is isolated from all others. As such, without loss of generality, consider how the labels are affected by critical points on one torus $T_k$:

\subsection{Maxima and minima}

At the extremum critical points, curves that are inessential on the torus $T_k$ are added or removed without changing existing curves. Consider the effect of such a move on the one-sided surface:

Suppose $t_i$ is a minimum and the image of $K$ has a single new inessential curve of intersection $\sigma$ with the torus $T_k$ in $(t_i, t_{i+1})$. Since $\sigma$ is inessential on $T_k$ and $K$ is \gi{}, $\sigma$ bounds discs $d \subset T_k$ and $D \subset K$.

Surger $K$ along $\sigma$, gluing parallel copies of $d$ to $K \setminus \sigma$ on either side of $T_k$, to get $\bar{K}$ and a $2$-sphere $d \cup D$. Since $M$ is irreducible, each $2$-sphere bounds a $3$-cell in $M$. The surfaces $K, \bar{K}$ differ only by the surgery to remove the trivial $2$-sphere $d \cup D$. Consider the potential boundary compressions of $K$ and $\bar{K}$ across $T_k$:

Consider the curve $\sigma$ on $K$. There may be additional inessential curves of intersection $\sigma_1, \sigma_2, \ldots, \sigma_m$ between $d$ and $D$. Partially order the set according to nesting within $\sigma_0 = \sigma$, where $\sigma_m$ is innermost. Since $K$ is \gi{} and $\sigma_m$ is innermost, $\sigma_m$ bounds discs $d_m$, $D_m \subset K$, the interiors of which are disjoint from $T_k$. Therefore, an orientable boundary compression of $K$ removes $\sigma_m$. Repeating this argument, working from innermost curves outwards, results in a series of orientable boundary compressions that removes the set $\{\sigma_k\}$. As such moves do not affect any others nontrivially and $K$ is isotopic to $\bar{K}$ in $M$, with the only difference between the surfaces being the nest of intersections $\{\sigma_k\}$, after orientable compressions, all such intersections are removed and $K$ is isotopic to $\bar{K}$ in $M_0$. Therefore, the two surfaces share all further boundary compressions and compress to the same boundary incompressible surface in $M_0$. Hence, the label does not change at $t_i$.
If $t_j$ is a maximum, $(t_{j-1}, t_j)$ has more inessential curves of intersection than $(t_j, t_{j+1})$. Applying the above argument, with surgery on the former, rather than the latter, interval, determines that the labels are again consistent across the two intervals separated by a maximum. Therefore, the extremum critical points do not change the labelling.

\subsection{Saddle points}

At saddle points, the existing curves of intersection are altered in one of a number of possible ways. This can manifest as additional (fewer) curves, a change in type, or a change of slope. The precise nature of the move will determine the result. Treat each case separately, where curves are considered as (in)essential on the torus $T$:
\vspace{\stretch{1}}
\pagebreak

\hspace{0.05\textwidth}\parbox[h]{0.9\textwidth}{
\textbf{(1)} An inessential curve is joined to, or split off, an inessential curve;\\
\vspace{-3mm}\\
\textbf{(2)} An inessential curve is joined to, or split off, an essential curve;\\
\vspace{-3mm}\\
\textbf{(3)} An inessential curve is joined to itself non-trivially, creating two essential curves;\\
\vspace{-3mm}\\
\textbf{(4)} An essential curve is joined to itself, or a parallel curve, along a compressing arc, rendering the curve(s) inessential;\\
\vspace{-3mm}\\
\textbf{(5)} An essential curve is joined to itself non-trivially, changing its slope.\\
}

Note that since all essential curves on each $T_k$ are parallel, any arc joining a pair of such curves lies in an annular region of the torus. As such, Case (4) is the only eventuality from joining essential curves. Indeed, this is the limiting factor restricting the argument to a toroidal boundary, rather than expanding to higher genus components.

\paragraph{Cases (1) and (2)}

If a saddle move adds or removes an inessential curve, the same argument as used for a maximum or minimum can be used to show that such a critical point does not change the labelling.

\paragraph{Case (3)}

Consider a region $(t_{i-1}, t_i)$, where a saddle move at $t_i$ joins an inessential curve $\sigma$ to itself, creating two parallel essential curves $\bar{\sigma}_1, \bar{\sigma}_2$ that bound an annulus $A \subset T_k$. Notice that $\sigma$ is not nested within other inessential curves, as this would preclude it from extending along an essential curve. However, the original curve $\sigma$ may contain nested curves. Since $\sigma$ is inessential on $T_k$ and $K$ is \gi{}, $\sigma$ bounds discs $d \subset T_k$ and $D \subset K$; in the event of nesting, $D$ crosses $T_k$. All curves $\sigma_1, \sigma_2, \ldots, \sigma_l$ nested within $\sigma$ are inessential and unchanged by the move at $t_i$, therefore can be removed by disc compressions, as per the Maxima and Minima argument. Let $\bar{K}$ be the surface after this disc surgery and $\bar{D}$ be the image of $D$, which is thus contained entirely in $M_0$ (or $M_{T_k}$).

By irreducibility, the $2$-sphere $d \cup \bar{D}$ bounds a $3$-cell in $M_0$ (or $M_{T_k})$ that is disjoint from $\bar{K} \setminus \bar{D}$. The move at $t_i$ extrudes $\bar{D}$, and the cell it bounds, along an essential slope, to become a boundary parallel annulus $A^{\prime} \subset K$, bounded by $\bar{\sigma}_1, \bar{\sigma}_2$.

Surger $K$ along $\bar{\sigma}_1, \bar{\sigma}_2$, gluing parallel copies of $A$ to $K \setminus \{\bar{\sigma}_1, \bar{\sigma}_2\}$ on either side of $T_k$ to get $\hat{K}$ and a torus $\bar{T}$ bounded by $A, A^{\prime}$. Since $A^{\prime}$ is boundary parallel, $\bar{T}$ bounds a solid torus in one of $M_0, M_{T_k}$ and the annular compression does not affect geometric compressibility. Hence, $\bar{T}$ can be disregarded and $\hat{K}$ is \gi{}. Notice that neither the essential curves, $\bar{\sigma}_1, \bar{\sigma}_2$, nor inessential curve, $\sigma$, of intersection remain between $\hat{K}$ and $T_k$. 

Let $\hat{K}^{\prime}$ be $\bar{K}$ after a disc compression along $\bar{D}$ to remove the original curve $\sigma$. Given the relationship between $\sigma$ and $\bar{\sigma}_1, \bar{\sigma}_2$, the annular compression to remove $A^{\prime}$ is equivalent to the disc compression to remove $\bar{D}$. Hence, $\hat{K}^{\prime}$ is isotopic to $\hat{K}$ in $M_0$. Therefore, in $(t_i, t_{i+1})$, the surface after annular compression intersects $M_0$ in the same components as if the disc compression were performed on $(t_{i-1}, t_i)$. However, it is known that such a compression does not change the labelling (by the Maxima and Minima argument), therefore  $(t_i, t_{i+1})$ contains the same components of $K$ as an interval labelled the same as $(t_{i-1}, t_i)$. Hence, $(t_i, t_{i+1})$ has the same label as $(t_{i-1}, t_i)$.

\paragraph{Case (4)}

Consider a pair of essential curves that are parallel on some $T_k$, but not on $K$. Since only saddles in Case (3) can introduce essential curves and such must be parallel on $K$ to existing curves, any non-parallel curves must exist at an endpoint $K \times \{0\} = K_0$ or $K \times \{1\} = L_0$.

Suppose at least one of the boundary incompressible components $K_0, L_0$ has more than one boundary component on $T_k$. Then on $K$ the complement $K \setminus K_0$ (or $K \setminus L_0$) consists entirely of discs in $M_{T_k}$, otherwise $K$ would not be embedded. Since such a surface is boundary incompressible from either side of $T_k$, the isotopy can only act trivially on $K$ at this component of the boundary; introducing curves at a minimum or as in Cases (1) and (2), which can be removed by disc and annular compressions. Specifically, Case (4) does not occur in this instance.

Consider the case where both $K_0, L_0$ have a single boundary component on $T_k$, thus any two essential curves of intersection are parallel on both $T_k$ and $K$. Let $t_i$ be a saddle point at which two essential curves in $(t_{i-1}, t_i)$ are joined to form a single inessential loop in $(t_i, t_{i+1})$. Reverse the argument used in Case (3), applying an annular compression to $(t_{i-1}, t_i)$ and a disc compression to $(t_i, t_{i+1})$ to establish that $t_i$ does not change the labelling.

\paragraph{Case (5)}

Suppose a saddle point $t_i$ joins an essential curve $\gamma$ to itself, changing its slope on $T_k$. This indicates a M\"obius band of $K$ being boundary compressed into, or out of, the solid filling torus. Note that in order for $K$ to remain embedded, such a move occurs only when there is a single essential curve of intersection between $K$ and $T_k$.

If $(t_{i-1}, t_i)$ is labelled $\mathcal{K}$, in this interval the image of $K$ restricts in $M_0$ to $K_0$ plus additional M\"obius bands in a collar of $T_k$. Whether a M\"obius band is added or removed in $M_0$, the resulting surface is again $K_0$ with additional M\"obius bands near $T_k$. As such, it compresses to $K_0$. However, by Proposition~\ref{prop:boundary_incomp}, a bounded one-sided surface boundary compresses to a unique boundary incompressible surface, so it cannot also compress to $L_0$. Therefore, $(t_i, t_{i+1})$ again bears the $\mathcal{K}$ label. Note, however, that there is a change in slope of the non-trivial curve of intersection, which is characteristic of a M\"obius band boundary compression.
\vspace{\baselineskip}

Having thus examined the potential effects of all types of critical point, it can be concluded that while the slope of essential curves may change, no critical point can affect a change of label from $\mathcal{K}$ to $\mathcal{L}$. As such, for the isotopy $\varphi$ to exist, $K_0$ must in fact be isotopic to $L_0$ in $M_0$, thus providing a contradiction.
\end{proof}

\section{One-sided Heegaard splittings of Figure 8 knot space}

In \oshs{} theory, \gi{} splitting surfaces are of key interest, forming, as they do, the basis for any splitting of a non-Haken $3$-manifold~\cite{nonhaken}. As such, having established a strict limitation on the behaviour of \gi{} one-sided surfaces in fillings of knot and link spaces, there are immediate consequences for \oss{s} of such manifolds. Specifically, \oss{} surfaces `built' from different boundary incompressible surfaces in the link space yield non-isotopic splitting surfaces in the filled manifold.

Using the constructions of Bartolini-Rubinstein~\cite{nonhaken} to produce \oss{} surfaces in even fillings of Figure 8 knot space, it becomes possible to differentiate between multiple surfaces of minimal genus. Thus completing the classification of all \oshs{s} for this infinite class of closed hyperbolic $3$-manifolds: a result without precedent for \tss{s}.

\begin{corollary}
In $(2p, q)$ Dehn fillings of Figure 8 knot space, where $p, q \in \mathbb{Z}$, $|p| > |q|$, $|p| > 2$ and $(p, q) = 1$, if $|{2p \over q}| <3$, then the two minimal genus \oss{} surfaces arising from the $(0, 1)$ fibre torus and one of the $(4, \pm1)$ punctured Klein bottles are non-isotopic in the filled manifold.
\end{corollary}

\begin{proof}
Such Dehn fillings of Figure 8 knot space have at most two non-isotopic \gi{} \oshs{s}, as established in~\cite{nonhaken}. It is known that uniqueness can be determined when $|{2p \over q}| >3$. However, there are potentially two distinct surfaces when the surfaces arising from the $(0, 1)$ torus and one of the $(4, \pm1)$ Klein bottles have the same, minimal genus. Given such surfaces comprise of distinct boundary incompressible surfaces in the knot space, Theorem~\ref{thm:unique_position} determines that they are indeed non-isotopic in the filled manifold.
\end{proof}

\paragraph{Example}

Consider  the $(8, 3)$ filling of Figure 8 knot space. Take the closed one-sided splitting surfaces $K_{(0, 1)}, K_{(4, 1)}, K_{(4, -1)}$ that arise from the \gi{}, boundary incompressible surfaces in the knot space, bounded by the curves $(0, 1)_K = (-8, 3)_T, (4,  1)_K = (4, -1)_T, (4, -1)_K = (20, -7)_K$ respectively. Using the continued fractions method of Bredon and Wood~\cite{bredon-wood} for the completions in the solid filling torus, the (non-orientable) genera can be calculated to be $3, 3, 5$ respectively.

The non-minimal genus surface $K_{(4, -1)}$ is known to geometrically compress to $K_{(0, 1)}$ by~\cite{nonhaken}. Furthermore, Theorem~\ref{thm:unique_position} determines that the two minimal genus surfaces $K_{(0, 1)}, K_{(4, 1)}$ are non-isotopic. Combined with the fact that \gc{} \oss{s} of non-Haken $3$-manifolds are stabilised, also by~\cite{nonhaken}, any \oss{} of the $(8, 3)$ filling of Figure 8 knot space is isotopic to one of $K_{(0, 1)}, K_{(4, 1)}$, or a stabilisation thereof.

\bibliographystyle{habbrv}
\bibliography{oshs}

\textit{Department of Mathematics and Statistics\\
The University of Melbourne\\
Parkville VIC 3010\\
Australia}

Email: \texttt{L.Bartolini@ms.unimelb.edu.au}

\end{document}